\documentclass[11pt]{article}

\title{\bf On the Dual of the Solvency Cone}

\author{Andreas L{\"o}hne \thanks{Martin-Luther-Universit{\"a}t Halle-Wittenberg, Department of Mathematics, 06099 Halle(Saale), Germany, andreas.loehne@mathematik.uni-halle.de}\and Birgit Rudloff \thanks{Princeton University, Department of Operations Research and Financial Engineering \& Bendheim Center for Finance, Princeton, NJ 08544, USA, brudloff@princeton.edu, research supported by NSF award DMS-1007938.}}

\usepackage{amsmath, amssymb, a4wide, amsthm, algorithm, algorithmic, tkz-graph}
\usepackage[pdftex]{hyperref}

\newtheorem{theorem}{Theorem}
\newtheorem{corollary}[theorem]{Corollary}
\newtheorem{lemma}[theorem]{Lemma}
\newtheorem{proposition}[theorem]{Proposition}
\newtheorem{example}[theorem]{Example}
\newtheorem{remark}[theorem]{Remark}

\newcommand{\of}[1]{\ensuremath{\left( #1 \right)}}
\newcommand{\cb}[1]{\ensuremath{ \left\{ #1 \right\} }}
\newcommand{\R}{\mathbb{R}}
\newcommand{\N}{\mathbb{N}}

\newcommand{\TT}{\mathcal{T}}
\newcommand{\PP}{\mathcal{P}}
\newcommand{\NN}{\mathcal{N}}
\newcommand{\cone}{{\rm cone\,}}
\newcommand{\conv}{{\rm conv\,}}
\newcommand{\Int}{{\rm int\,}}
\newcommand{\st} {\ensuremath{|\;}}

\begin{document}

\maketitle

\begin{abstract} \noindent
A solvency cone is a polyhedral convex cone which is used in Mathematical Finance to model proportional transaction costs. It consists of those portfolios which can be traded into nonnegative positions. In this note, we provide a characterization of its dual cone in terms of extreme directions and discuss some consequences, among them: (i) an algorithm to construct extreme directions of the dual cone when a corresponding ``contribution scheme'' is given; (ii) estimates for the number of extreme directions; (iii) an explicit representation of the dual cone for special cases. 

The validation of the algorithm is based on the following easy-to-state but difficult-to-solve result on bipartite graphs: Running over all spanning trees of a bipartite graph, the number of left degree sequences equals the number of right degree sequences. 

\medskip


\medskip


\medskip

\noindent The final publication in Discrete Applied Mathematics, DOI: 10.1016/j.dam.2015.01.030,
is available via \url{http://dx.doi.org/10.1016/j.dam.2015.01.030}.

\end{abstract}

\noindent
We investigate the structure of a polyhedral convex cone, which in Mathematical Finance is called {\em solvency cone} \cite{Kabanov99, KabSaf, Schachermeyer}. Consider a portfolio $x \in \R^d$ given in physical units of $d \geq 2$ assets (or currencies) and assume that we are given market prices $\pi_{ij} > 0$ saying that, for any $z \geq 0$, $\pi_{ij} z$ units of asset $i$ can be transferred into $z$ units of asset $j$. An important special case is $\pi_{ij} = a_j / b_i$, where $a_j$ is the ask price (per unit) of asset $j$ and $b_i$ is the bid price (per unit) of asset $i$, both prices expressed by a certain reference currency (num\'eraire). A portfolio $x$ is {\em solvent} if it can be transferred into a portfolio with only nonnegative components. Under suitable axioms to the market prices, the set of all solvent portfolios provides a polyhedral convex cone with nice properties --- the solvency cone $K_d$. Among other results, we solve a problem stated in 2000 by Bouchard and Touzi \cite[page 706]{BouTou00}: {\em ``provide explicitly a generating family for the polar [or dual] cone [of  $K_d$ for $d>2$].'' }

We present the concepts and results about the solvency cone independently of this interpretation, but we insert some remarks and examples which might be useful for readers from {\em Mathematical Finance}. The solvency cone is closely related to generalized optimal flow problems, which is a well-studied problem of {\em Combinatorial Optimization} \cite{CCPS98,Jewell62}. The proof of the main result (Theorem \ref{theorem_7}) about the existence (and construction) of certain extreme directions of the dual of the solvency cone requires statements from {\em Graph Theory}. Using a link to {\em Algebraic Combinatorics} (provided by Sang-il Oum), it is shown that, running over all spanning trees of a bipartite graph, the number of left degree sequence equals the number of right degree sequences. 

Let $d \in \cb{2,3,\dots}$, $V=\cb{1,\dots,d}$ and let $\Pi=(\pi_{ij})$ be a ($d \times d$)-matrix of real numbers such that
\begin{align}
   \label{equation_1}\forall i \in V:&\quad \pi_{ii}=1,\\
   \label{equation_2}\forall i,j \in V:&\quad 0<\pi_{ij},\\
   \label{equation_3}\forall i,j,k \in V:&\quad \pi_{ij} \leq \pi_{ik}\pi_{kj},\\
   \label{equation_4}\exists i,j,k \in V:&\quad \pi_{ij} < \pi_{ik}\pi_{kj}.
\end{align}
In a few situations, only if explicitly mentioned, \eqref{equation_3} and \eqref{equation_4} will be replaced by
\begin{equation}   
   \label{equation_5}\forall i,j \in V,\;\forall k \in V \setminus \cb{i,j}:\quad \pi_{ij} < \pi_{ik}\pi_{kj}.
\end{equation}
The polyhedral convex cone
$$ K_d:=\cone \cb{\pi_{ij} e^i - e^j \st ij \in V\times V}:=\cb{\textstyle\sum_{ij\in V\times V} z_{ij}(\pi_{ij} e^i - e^j) \,\big|\; z \in \R^{d\times d},\, z\geq 0}$$
is called {\em solvency cone} induced by $\Pi$, cf. \cite{Kabanov99}.

We denote by $K_d^+:=\cb{y \in \R^d \st \forall x \in K_d:\ x^\top y \geq 0}$ the (positive) dual cone of $K_d$. The generating vectors $\pi_{ij} e^i - e^j$ of $K_d$ induce an inequality representation of $K_d^+$. We start with some well-known statements, see e.g. \cite{BouTou00,Kabanov99,Schachermeyer}. 
\begin{proposition} \label{proposition_1} 
The dual cone of $K_d$ can be expressed as
\begin{equation}\label{equation_6}
K_d^+=\cb{y \in \R^d \st \forall i,j \in V: \pi_{ij} y_i \geq y_j}.
\end{equation}
\end{proposition}
\begin{proof}
Let $M$ denote the set on the right hand side of \eqref{equation_6}. Let $y \in M$, then $x^\top y \geq 0$ for all $x \in K_d$,  i.~e. $M \subseteq K_d^+$. Conversely, let $y \in K_d^+$. Assuming that $y \not\in M$, we obtain $i,j \in V$ such that $\pi_{ij}y_i < y_j$. For $x=\pi_{ij} e^i - e^j\in K_d$, this means $x^\top y < 0$, a contradiction.
\end{proof}

\begin{proposition}\label{proposition_2} One has $\R^d_+ \setminus\cb{0} \subseteq \Int K_d$ and $K_d^+ \setminus \cb{0} \subseteq \Int \R^d_+$.
\end{proposition}
\begin{proof}
By \eqref{equation_4} we have $\pi_{kj}\pi_{ji}>\pi_{ki}$ for some $i,j,k \in V$. Using \eqref{equation_3}, we get $\pi_{ki}\pi_{ij}\pi_{ji} \geq \pi_{kj}\pi_{ji}>\pi_{ki}$ which implies $\pi_{ij}\pi_{ji}>1$.
We have $(\pi_{ij} e^i - e^j) + 1/\pi_{ji} (\pi_{ji} e^j - e^i) =(\pi_{ij}-1/\pi_{ji}) e^i$, whence $e^i \in K_d$ for some $i \in V$. For arbitrary $j\in V$ we obtain
$e^i + (\pi_{ji} e^j - e^i) = \pi_{ji} e^j \in K_d$. By \eqref{equation_2}, $e^j\in K_d$, i.~e. every unit vector belongs to $K_d$. Hence $\R^d_+ \subseteq K_d$ and $K_d^+ \subseteq \R^d_+$.

Assume there is $y \in K_d^+$ with $y_i=0$ for some $i \in V$. The inequalities in \eqref{equation_6} together with \eqref{equation_2} imply $y=0$. Hence $K_d^+ \setminus \cb{0} \subseteq \Int \R^d_+$. 

Let $x \in \R^d_+\setminus\cb{0}$ and assume that $x \not\in \Int K_d$. By a typical separation argument there is $y \in K_d^+\setminus\cb{0} \subseteq \Int \R^d_+$ such that $x^\top y \leq 0$. This implies $x=0$, a contradiction.
\end{proof}

\begin{remark} Condition \eqref{equation_4} can be omitted if the definition of the solvency cone is slightly amended, for instance, $\tilde K_d:=K_d + \R^d_+$, compare \cite[page 22]{Schachermeyer}. 	
Then, the first part of the proof of Proposition \ref{proposition_2} becomes obsolete. 
Condition \eqref{equation_4} is only used to prove Proposition \ref{proposition_2}. In terms of Mathematical Finance, condition \eqref{equation_4} excludes the trivial case of no transaction costs.
\end{remark}

\begin{proposition}\label{proposition_2a}
One has $K_d \cap -\R^d_+ = \cb{0}$.
\end{proposition}
\begin{proof}
Set $y=(\pi_{11},\pi_{12},\dots,\pi_{1d})^T$, then $y \in K_d^+ \setminus \cb{0}$ by \eqref{equation_2} and \eqref{equation_3}. Assume there is some nonzero $x \in K_d \cap -\R^d_+$. By Proposition \ref{proposition_2}, $x \in -\Int K_d$. It follows that $0 = x - x \in K_d + \Int K_d \subseteq \Int K_d$. Hence $K_d = \R^d$ and $K^d_+=\cb{0}$, a contradiction.
\end{proof}

Let us recall some standard concepts related to digraphs. A {\em digraph} (or {\em directed graph}) $G=(V,E)$ is a pair $(V,E)$, where $V=V(G)$ is a finite set of {\em nodes} and $E = E(G) \subseteq V\times V$ is the set of {\em arcs}. For an arc $(i,j)\in E$ we also write $ij \in E$ for short.  All digraphs are assumed to be {\em simple}, i.~e., there are neither multiple arcs nor loops (arcs of the form $ii$). A {\em path} in $G$ is a sequence $(n_1,a_1,n_2,a_2,\dots,a_{k-1}, n_k)$ of pairwise distinct nodes $(n_1,n_2,\dots,n_k)$ and arcs $(a_1,a_2,\dots,a_{k-1})$ of $G$ such that $a_i = n_i n_{i+1}$ (called {\em forward arc}) or $a_i=n_{i+1} n_i$ (called {\em backward arc}) for $i \in \cb{1,\dots,k-1}$. If for $k\geq 3$, $n_1=n_k$ is allowed in the latter definition, we speak about a {\em cycle} in $G$. A digraph $G$ is said to be {\em connected} if there is a path in $G$ between any two distinct nodes. A digraph $H$ is called a {\em subgraph} of the digraph $G$ if $V(H)\subseteq V(G)$ and $E(H) \subseteq E(G)$. A {\em spanning tree} of a digraph $G$ is a connected subgraph of $G$ with node set $V(G)$ and having no cycles. The degree $\deg_G(i)$ is the number of arcs of a digraph $G$ which are {\em incident}  to $i \in V=V(G)$ (i.~e. of the form $ij$ or $ji$ for $j \in V$). If any arc $ij$ of $G$ is identified with $ji$, $G$ is called {\em (undirected) graph} and $ij$ is called an {\em edge} of $G$. The above concepts are defined likewise, see e.g. \cite{BangJensenGutin2010, CCPS98} for further details.  

The set $V=\cb{1,\dots,d}$ is now splitted into two nonempty disjoint sets $P$ and $N$; the pair $(P,N)$ is called a {\em bipartition} of $V$. 
We speak about a {\em bipartite digraph} $G=(V,E)$ if $V=P\cup N$ and $E \subseteq (P \times N)\cup (N \times P)$ for a bipartition $(P,N)$ of $V$. In particular, we denote by $G=G(P,N)$ the bipartite digraph with node set $V$ and arc set $E=P\times N$. Given a bipartition $(P,N)$ of $V$, a vector $y \in \R^d$ is said to be {\em generated by a tree $T$} if $T$ is a spanning tree of $G(P,N)$ such that
\begin{equation}\label{equation_7}
  \forall ij \in E(T) \subseteq P \times N:\; \pi_{ij} y_i = y_j > 0,
\end{equation}
see Figure \ref{fig:1} (right) for an illustration. A vector $y \in \R^d$ is called {\em feasible} (with respect to $(P,N)$) if 
\begin{equation}\label{equation_8}
  \forall ij \in P \times N:\pi_{ij} y_i \geq y_j > 0.
\end{equation}
If $y \in \R^d$ is both generated by a tree $T$ and feasible, we say $y$ is a {\em feasible tree solution} (with respect to $(P,N)$). 

\begin{figure}
\begin{center}
\begin{tikzpicture}
  \Vertex[x=5 ,y=5]{1};
  \Vertex[x=5 ,y=3]{4};
  \Vertex[x=8 ,y=6]{2}
  \Vertex[x=8 ,y=4]{3}
  \Vertex[x=8 ,y=2]{5}
  \tikzset{EdgeStyle/.style={->,relative=false,in=160,out=50}}
  \Edge[label=$\pi_{12}$](1)(2)
  \tikzset{EdgeStyle/.style={->,relative=false,in=170,out=-20}}
  \Edge[label=$\pi_{13}$](1)(3)
  \tikzset{EdgeStyle/.style={->,relative=false,in=170,out=-60}}
  \Edge[label=$\pi_{15}$](1)(5)
  \tikzset{EdgeStyle/.style={->,relative=false,in=190,out=60}}
  \Edge[label=$\pi_{42}$](4)(2)
  \tikzset{EdgeStyle/.style={->,relative=false,in=190,out=20}}
  \Edge[label=$\pi_{43}$](4)(3) 
  \tikzset{EdgeStyle/.style={->,relative=false,in=200,out=-50}}
  \Edge[label=$\pi_{45}$](4)(5)
  \Vertex[x=12 ,y=5]{1};
  \node[left] at (1.180) {$y_1=1$};
  \Vertex[x=12 ,y=3]{4};
  \node[left] at (4.180) {$y_4=\frac{\pi_{13}}{\pi_{43}}$};
  \Vertex[x=15 ,y=6]{2}
  \node[right] at (2.0) {$y_2=\pi_{12}$};
  \Vertex[x=15 ,y=4]{3}
  \node[right] at (3.0) {$y_3=\pi_{13}$};
  \Vertex[x=15 ,y=2]{5}
  \node[right] at (5.0) {$y_5=\pi_{15}$};  
  \tikzset{EdgeStyle/.style={->,relative=false,in=160,out=50}}
  \Edge(1)(2)
  \tikzset{EdgeStyle/.style={->,relative=false,in=170,out=-20}}
  \Edge(1)(3)
  \tikzset{EdgeStyle/.style={->,relative=false,in=170,out=-60}}
  \Edge(1)(5)
  \tikzset{EdgeStyle/.style={->,relative=false,in=190,out=20}}
  \Edge(4)(3)
\end{tikzpicture}
\end{center}
\caption{Left: The digraph $G(P,N)$ for $P=\cb{1,4}$ and $N=\cb{2,3,5}$, where the arcs $ij \in P \times N$  are associated with market prices $\pi_{ij}$. Right: A spanning tree $T$ of $G(P,N)$ and  a vector $y=(y_1,y_2,y_3,y_4,y_5)^\top$ generated by $T$.}
\label{fig:1}
\end{figure}
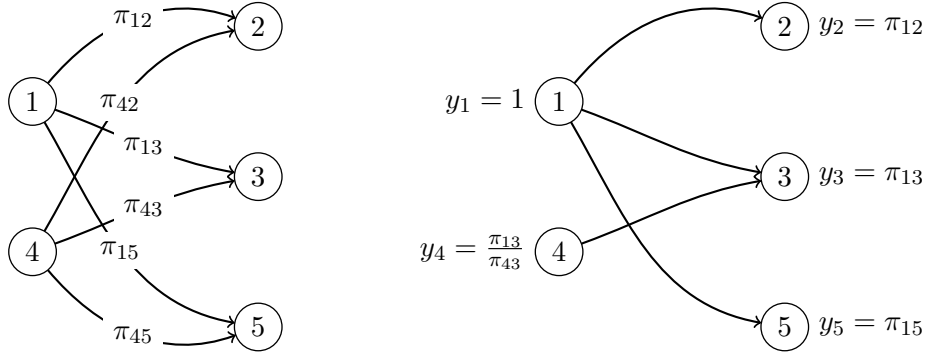

Let us provide an interpretation in terms of Mathematical Finance. Consider a portfolio $x \in \R^d$ given in physical units of $d \geq 2$ assets (or currencies) with at least one positive position $x_k>0$ and at least one negative position $x_l<0$. Setting $P:=\cb{i \in V \st x_i \geq 0}$ and $N:=\cb{j \in V \st x_j < 0}$, we obtain a bipartition. An arc $ij$ of the bipartite digraph $G(P,N)$ stands for a transaction of an unknown amount $z_{ij} \geq 0$ of asset $i$ into asset $j$ with respect to the market price $\pi_{ij}$, i.~e., $\pi_{ij} z_{ij}$ units of asset $i$ are transferred into $z_{ij}$ units of asset $j$. The variable $z_{ij}\geq 0$ can be seen as a flow along the arc $ij$ with a gain $\pi_{ij}^{-1}$. This leads to a generalized network flow problem (see e.g. \cite{Jewell62}) with demands $h_i=x_i$ ($i \in P$), $h_j\leq x_j$ ($j \in N$), a feasible flow of which ensures that $x$ is transferred into a portfolio with only nonnegative positions. The digraph $G(P,N)$ represents all potentially useful transactions to achieve this aim. 

A vector $y \in K_d^+\setminus\cb{0}$ describes a price system which is consistent in the sense that every portfolio $x \in K_d$ has a nonnegative value with respect to $y$, that is, $y^\top x \geq 0$. We have $\pi_{ij} \geq y_j / y_i$ ($i,j \in V$) for all those price systems, see \eqref{equation_6}. But, of course, a transaction is realizable (in the sense that the price system $y$ is compatible to market prices $\pi_{ij}$) only in case of $\pi_{ij} = y_j / y_i$, compare \eqref{equation_7}. A feasible tree solution 
$y$ generated by a tree $T$ describes a consistent price system where transactions along arcs of the tree $T$ are realizable. Theorem~\ref{theorem_3} below states that these price systems are exactly those which cannot be expressed as a non-trivial combination of two other consistent price systems.

\begin{example} \label{ex:1}
 Consider 5 assets represented by the set $V=\cb{1,2,3,4,5}$ and a corresponding portfolio $x=(x_1,x_2,x_3,x_4,x_5)^\top$ where we assume that $x_1,x_4 \geq 0$ (long positions) and $x_2,x_3,x_5 < 0$ (short positions). We consider in this situation the bipartition $(P,N)$ with $P=\cb{1,4}$ and $N=\cb{2,3,5}$ and the digraph $G(P,N)$, which is illustrated in Figure \ref{fig:1} (left). The subgraph $T$ with nodes $V$ and arcs $\cb{12,13,15,43}$ provides a spanning tree, see Figure \ref{fig:1} (right). The vector $y=(1,\pi_{12},\pi_{13},\pi_{13}\pi_{43}^{-1}, \pi_{15})$ is generated by the tree $T$. It is a feasible tree solution if \eqref{equation_8} holds along the arcs $42$ and $45$, which are missing in $T$ in comparison with $G(P,N)$, that is, $\pi_{42} \pi_{13} \geq \pi_{12} \pi_{43}$ and $\pi_{45}  \pi_{13} \geq \pi_{15} \pi_{43}$.
\end{example}

The next result relates feasible tree solutions to extreme directions of $K_d^+$. Subsequently, directions (and likewise feasible tree solutions) $y,z \in \R^d \setminus \cb{0}$ with $y=\alpha z$ for some $\alpha > 0$ are considered to be identical. Recall that $y \in K \setminus \cb{0}$ is called {\em extreme direction} of a polyhedral convex cone $K \subseteq \R^d$ if it cannot be expressed as the sum of two other directions of $K$. Equivalently, $y \in \R^d \setminus \cb{0}$ is an extreme direction of $K$ if it belongs to an edge of $K$ (i.~e. a $1$-dimensional face). For further details on polyhedral cones, see e.g. \cite{Rockafellar72, Webster94}.

To further motivate the next result, consider the following task: One wants to rearrange a given portfolio $x\in\R^d$ (denoted in physical units) according to current market prices $\Pi=(\pi_{ij})$ with the aim to sell certain assets and buy certain other assets. This determines a bipartition: if the holding in asset $i$ is supposed to increase, collect index $i$ into the set $N$, if asset $j$ is supposed to decrease or stay the same, collect $j$ into the set $P$. Then, the question is, which assets can be obtained from $x$ when trading in a smart way, i.~e. without burning unnecessarily money. The set of all portfolios that $x$ can be exchanged into is given by $x-K_d$, but for a smart trade we are only interested in portfolios in the boundary of $x-K_d$. The boundary of $K_d$ is the union of its facets $F(y)= \cb{x \in K_d :\; y^\top x = 0}$ taken over all extreme direction $y$ of the dual cone $K_d^+$. Theorem \ref{theorem_3} tells us that these extreme directions correspond to feasible tree solutions with respect to some bipartition. The arcs $ij$ of the corresponding spanning tree $T$ show which assets $i\in P$ are exchanged into which assets $j\in N$. Later we will see that there are at least one and at most $\binom{d-2}{p-1}$ (where $p:=|P|$) feasible tree solutions for any bipartition $(P,N)$, compare Theorem \ref{theorem_7} and Corollary \ref{corollary_10}.

\begin{theorem}\label{theorem_3}
For $y \in \R^d$, the following statements are equivalent.
\begin{itemize}
\item[(i)] $y$ is an extreme direction of $K_d^+$;
\item[(ii)] $y$ is a feasible tree solution with respect to some bipartition $(P,N)$ of $V$.
\end{itemize}
\end{theorem}
\begin{proof}
(i) implies (ii). An extreme direction $y$ is a nonzero vector in $K_d^+$, hence it satisfies the inequalities in \eqref{equation_6} and, by Proposition \ref{proposition_2}, $y$ has positive components. Consequently, \eqref{equation_8} is satisfied for every bipartition $(P,N)$ of $V$. The extreme directions of a polyhedral cone in $\R^d$ correspond to its edges. It is well-known (e.g. from linear programming) that $y\in K_d^+$ belongs to an edge of $K_d^+$ if and only if there are $d-1$ linearly independent inequalities in \eqref{equation_6} being satisfied with equality. Hence there is a matrix 
$$ M=\cb{\pi_{i_1j_1} e^{i_1} - e^{j_1},\dots,\pi_{i_{d-1}j_{d-1}} e^{i_{d-1}} - e^{j_{d-1}}} $$
with linearly independent columns such that $M^\top y = 0$. Define a digraph $T=(V,E)$ with node set $V=\cb{i_1,\dots,i_{d-1},j_1,\dots,j_{d-1}}$ (multiple occurrence of indices considered as one node) and arc set $E=\cb{(i_1,j_1),\dots,(i_{d-1},j_{d-1})}$. 

We show that $T$ has no cycle. Assume that $T$ has a cycle $C$ consisting of $n$ arcs (backward or forward arcs in $C$). Let $M_C$ be the corresponding $n \times n$ submatrix of $M$ (obtained by taking the columns corresponding to the arcs of $C$ and by deleting zero rows corresponding to nodes not belonging to $C$).  Let $y_C \in \R^n$ be the vector which arises from $y>0$ by deleting all components $y_i$ for nodes $i$ not belonging to $C$. Since the columns of $M$ are linearly independent, $M_C$ has rank $n$. But, $y_C > 0$ and $M_C^\top y_C=0$, a contradiction. Since $T$ has $d-1$ arcs but no cycle, there must be at least $d$ nodes. Clearly $T$ must have exactly $d$ nodes and, consequently, $T$ is connected. Thus $T$ is a spanning tree of the complete graph $G(V)$.

Set $P:=\cb{i_1,\dots,i_{d-1}}$, $N':=\cb{j_1,\dots,j_{d-1}}$, $N:=N'\setminus P$.
Clearly, $P$ and $N$ are nonempty and as shown above, we have $P \cup N = P \cup N' = V$ and $P\cap N = \emptyset$, i.~e. $(P,N)$ is a bipartition.

Let $j \in P \cap N'$. Then there exist $i \in P, k \in N'$ such that $ij, jk \in E(T)$. Since $T$ contains no cycles, we have $ik \not\in E(T)$. Replacing $ij$ by $ik$ we obtain again a spanning tree. We have $\pi_{ij} y_i = y_j$ and $\pi_{jk} y_j = y_k$.
Using \eqref{equation_6}, we obtain $\pi_{ij}\pi_{jk} y_i = y_k \leq \pi_{ik} y_i$. Since  $y>0$, we get $\pi_{ij}\pi_{jk} \leq \pi_{ik}$. By \eqref{equation_3} we obtain $\pi_{ij}\pi_{jk} = \pi_{ik}$ which implies $\pi_{ik} y_i = y_k$. Repeating this procedure we obtain after finitely many steps a spanning tree $S$ of the bipartite graph $G(P,N)$ such that $\pi_{ij} y_i = y_j$ whenever $ij \in E(S)$, i.~e. \eqref{equation_7} holds.

(ii) implies (i). We start showing that $y$ belongs to $K^+_d$, i.~e. $\pi_{ij} y_i \geq y_j$ for all $i,j \in V$. Let $i,j \in V$ be given. We distinguish four cases. Case (a). If $i \in P$ and $j \in N$, the statement follows from \eqref{equation_8}. Case (b). If $i \in P$ and $j \in P$, there exists $k \in N$ such that $jk$ belongs to the spanning tree $T$ of the bipartite graph $G(P,N)$, hence $\pi_{jk} y_j = y_k$. By \eqref{equation_8} we have $\pi_{ik} y_i \geq y_k$. Using \eqref{equation_3}, we get 
$$\pi_{ij} y_i \geq \frac{\pi_{ik}}{\pi_{jk}} y_i \geq \frac{y_k}{\pi_{jk}}= y_j.$$  
Case (c). If $i \in N$ and $j \in N$, there is $k \in P$ such that $ki$ belongs to $T$. We have $\pi_{ki} y_k = y_i$ and, by \eqref{equation_8}, $\pi_{kj} y_k \geq y_j$. It follows
$$\pi_{ij} y_i \geq \frac{\pi_{kj}}{\pi_{ki}} y_i = \pi_{kj} y_k \geq y_j.$$ 
Case (d). Let $i \in N$ and $j \in P$. There exists $k \in P$ such that $\pi_{ki} y_k = y_i$ and as shown in case (b) we have $\pi_{kj} y_k \geq y_j$. The result now follows like in case (c). 

The $d-1$ arcs of $T$ correspond to $d-1$ inequalities in \eqref{equation_6}, which are satisfied with equality. It remains to show that these $d-1$ equations are linearly independent (i.~e. their coefficient vectors are so). For $d=2$ the statement is obviously true. Let $d>2$. It is well-known that every spanning tree $T$ has a node $k$ which is incident to exactly one of its arcs. By \eqref{equation_2}, this means that the variable $y_k$ occurs in exactly one equation. If in this equation, the node $k$, and the arc incident to $k$ are omitted, we obtain the same problem for dimension $d-1$, i.~e. the desired statement follows by induction. 
\end{proof}

We next show the existence of feasible tree solutions $y$, where additional conditions to the corresponding spanning tree can be supposed. Let $\N= \cb{1,2,\dots}$ denote the positive integers. Let a bipartition $(P,N)$ of $V=\cb{1,\dots,d}$ be given. A vector $c \in \N^P$ is called {\em $P$-configuration} if $\sum_{i \in P} c_i = d-1$. Likewise we introduce an {\em $N$-configuration}. The {\em degree vector (or degree sequence)} $\deg_T(A) \in \N^A$ of a node set $A \subseteq V$ of a tree (or graph) $T$ is the vector with components $\deg_T(i)$, $i \in A$. Clearly, if $T$ is a spanning tree of $G(P,N)$, then $\deg_T(P)$ is a $P$-configuration, also called {\em left degree sequence} of $T$. Likewise, $\deg_T(N)$ is called {\em right degree sequence} of $T$. By $\TT(H)$ we denote the set of all spanning trees of a graph or digraph $H$.

\begin{example} \label{ex:2}
	Consider the bipartition $(P,N)$ and the spanning tree $T$ from Example \ref{ex:1}, see also Figure \ref{fig:1}. The set of all $P$-configurations is $\cb{(3,1),(2,2),(1,3)}$ and the set of all $N$-configurations is $\cb{(2,1,1),(1,2,1),(1,1,2)}$. For the tree $T$, the left degree sequence is $\deg_T(P) = (3,1)$, and the right degree sequence is $\deg_T(N) = (1,2,1)$.	
\end{example}

The following result states that the number of left degree sequences of $T$ equals the number of right degree sequences of $T$, when $T$ is running over all spanning trees of a graph (or digraph) $H$. For $H=G(P,N)$ (and other special cases) the result is easy to prove but the general case seems to be highly non-trivial. The proof is based on results by Postnikov \cite{Postnikov09} about generalized permutohedra. It is important to mention that the link to Postnikov's results was found by Sang-il Oum\footnote{The authors are also greatly indebted to Annabell Berger, Matthias Kriesell, Paul Seymour and Richard Stanley for their guidance and support with respect to the proof of Lemma \ref{lemma_4}, which had been the missing tool to prove the main result, Theorem \ref{theorem_7}.}.

\begin{lemma} \label{lemma_4}
Let $H$ be a bipartite graph with bipartition $(P,N)$. Then 
\[ |\{\deg_T(P) \st T \in \TT(H)\}|=|\{\deg_T(N) \st T \in \TT(H)\}|. \]
\end{lemma}
\begin{proof}
	This is a consequence of Theorem 11.3 and Corollary 11.8 of \cite{Postnikov09}. The details are discussed in the appendix of this article.
\end{proof}

\begin{lemma}\label{lemma_5}
Let $H=H(P,N)$ be a bipartite digraph such that $E(H) \subseteq P \times N$ and let $S,T$ be two spanning trees of $H$ such that $\deg_S(P)=\deg_T(P)$.  For every arc $ij \in E(T) \setminus E(S)$ there exists a cycle $C$ in $H$ such that every forward arc in $C$ belongs to $S$ and every backward arc in $C$ belongs to $T$ and $C$ contains the arc $ij$.\footnote{The authors thank Matthias Kriesell for providing a shorter proof of Lemma \ref{lemma_5}.}
\end{lemma}
\begin{proof}
We construct an auxiliary digraph $D$ on $P \cup N$, where there is an arc from
$k \in P$ to $l \in N$ if $k,l$ are connected by an edge in $S$ and an arc from $l \in N$
to $k \in P$ if $k,l$ are connected by an edge in $T$. A continuously directed cycle of length at least $3$ in $D$
then corresponds to an alternating cycle in $H$ as desired. 
Consequently, it suffices to prove that $D$ is strongly connected. 
Suppose this is not the case. By Menger's Theorem \cite{BangJensenGutin2010},
there exists a nonempty, proper subset $X$ of $V(D)$ such that there is no edge
from $X$ to $V(D) \setminus X$ in $D$. Let $s,t$ denote the number of edges in $S,T$, respectively, connecting two vertices in $X$,
and let $u$ denote the number of edges in $T$ connecting a vertex in $X$ to a vertex in $V(H) \setminus X$.
Let $v,w$ be the number of components of the subforest induced by $X$ in $S,T$, respectively.
It follows $v=|X|-s \geq 1$ and $w=|X|-t$. By construction of $D$
and since there is no edge from $X$ to $V(D) \setminus X$ in $D$,
we obtain $s=\sum\limits_{k \in P \cap X} \deg_D^-(k)$ and $t+u=\sum\limits_{k \in P \cap X} \deg_D^+(k)$,
where $\deg^-_D(k)$ and $\deg^+_D(k)$ denote the number of edges of $D$ starting in $k$
and terminating in $k$, respectively.
By the condition to the degrees of $S$ and $T$, the two sums on the right in these equations are equal.
Hence $s=t+u$. It follows $w=|X|-t=|X|-s+u=v+u>u$.
However, every component of the subgraph induced by $X$ in $T$ is connected
by at least one edge of $T$ to a vertex in $V(H) \setminus X$ --- which implies $u \geq w$,
contradiction.
\end{proof}

For a feasible tree solution $y$ we define a subgraph $H(y)$ of the bipartite digraph $G=G(P,N)$ by $V(H(y)) := V(G)$ and 
$E(H(y)) := \cb{ij \in P\times N \st \pi_{ij} y_i = y_j}$. We set
\[ \PP(y):=\{\deg_T(P) \st T \in \TT(H(y))\} \text{ and }  \NN(y):=\{\deg_T(N) \st T \in \TT(H(y))\}. \]

\begin{lemma}\label{lemma_6}
Let $x,y$ be two feasible tree solutions such that $x \neq \alpha y$ for all $\alpha>0$. Then
\begin{equation*}
	 \PP(x) \cap \PP(y) =\emptyset \quad \text{ and }\quad  \NN(x) \cap \NN(y) =\emptyset.	
\end{equation*}
\end{lemma}
\begin{proof}
First observe that $H(x) \cap H(y)$ does not contain a spanning tree of $G$. Assume there is a $P$-configuration $c \in \PP(x) \cap \PP(y)$. Then there are two spanning trees $S \in \TT(H(x))$ and $T \in \TT(H(y))$ such that $c=\deg_S(P)=\deg_T(P)$. Since $x \neq \alpha y$ for all $\alpha > 0$ we have $S\neq T$. 
Since $S \in \TT(H(x))$ and $\TT(H(x)) \cap \TT(H(y))=\emptyset$, there is an edge $ij\in E(S) \setminus E(H(y))$, i.~e. $\pi_{ij} y_i > y_j$. 
By Lemma \ref{lemma_5} there exists a cycle $C=\cb{e_1,\dots,e_{2k}}$ in $G$ such that $e_{2l-1} \in E(S)$ and $e_{2l} \in E(T) $ for $l=1,\dots,k$ and  $e_{2l-1}= ij$ for some $l\in\{1,\dots,k\}$. 
Taking into account $\pi_{ij} y_i > y_j$, the inequalities \eqref{equation_8}, and the equalities arising from $T \in \TT(H(y))$ for the feasible tree solution $y>0$ we obtain 
$\prod_{l=1}^k \frac{\pi_{2l-1}}{\pi_{2l}} > 1$. Now using for the same cycle the inequalities \eqref{equation_8} and the equalities arising from $S \in \TT(H(x))$ for the feasible tree solution $x>0$ we obtain the contradiction $\prod_{l=1}^k \frac{\pi_{2l}}{\pi_{2l-1}} \geq 1$.  The second statement can be proven likewise taking into account that the role of $P$ and $N$ can be commuted in Lemma \ref{lemma_5}. 
\end{proof}

We are now ready to prove our main result.

\begin{theorem}\label{theorem_7}
For every bipartition $(P,N)$ of $V$ and every $P$-configuration $c \in \N^P$ there exists a feasible tree solution $y \in \R^d$ generated by a spanning tree $T$ of the bipartite graph $G(P,N)$ with $\deg_T(P)=c$. An analogous statement holds if an $N$-configuration is given.
\end{theorem}

\begin{proof}
We prove the result by induction. For a bipartition of $V=\cb{1,2}$, the statement is obvious. Assume the statement holds for every bipartition of $V_0=\cb{1,\dots,d-1}$. Let $(P,N)$ be a bipartition of $V=\cb{1,\dots,d}$ and let $c\in \N^P$ be a $P$-configuration.

Set $p=|P|$ and $n=|N|$. For the moment we assume that $p \geq n$. Hence there exists $i \in P$ such that $c_i=1$, say $d \in P$ and $c_d = 1$. Set $P_0 := P \setminus\cb{d}$ and consider the bipartition $(P_0, N)$ of $V_0$ and the $P_0$-configuration $c_0 \in \N^P_0$ which arises from $c$ by omitting the last component. We assumed that there is a feasible tree solution $y^0 \in \R^{d-1}$ generated by a spanning tree $T_0$ of $G(P_0, N)$ with $\deg_{T_0}(P_0)= c_0$. Set $T:=(V_0 \cup \cb{d}, E(T_0) \cup \cb{dk})$, where $k \in \arg \max \{y^0_j / \pi_{dj} \,|\; j \in N\}$, and $y := (y^0_1,\dots y^0_{d-1}, y^0_k / \pi_{dk})^\top$. Of course, $T$ is a spanning tree of $G(P,N)$ and we have $\deg_T(P)= c$. Let $j \in N$ be given. We have $y_d = y_k / \pi_{dk}  \geq y_j / \pi_{dj}$ and hence $\pi_{dj} y_d \geq y_j$, i.~e. $y$ is feasible with respect to $(P,N)$. This means, if $p \geq n$, there exists a feasible tree solution $y$ with respect to a tree $T$ such that $\deg_T(P)=c$. 

Still assuming $p \geq n$, we next show that for every $N$-configuration $b$ there exist a feasible tree solution $y$ with respect to a tree $T$ such that $\deg_T(N)=b$. To this end, note first that there are $\binom{d-2}{p-1}$ different $P$-configurations and the same number of different $N$-configurations, namely $\binom{d-2}{n-1} = \binom{d-2}{(d-p)-1} = \binom{d-2}{(d-2)-(p-1)}= \binom{d-2}{p-1}$. Given a $P$-configuration $c$ we obtain the result for one corresponding $N$-configuration, namely, if $y$ is a feasible tree solution generated by $T$ such that $c=\deg_T(P)$, then we have the desired result for $b:=\deg_T(N)$. It remains to show that all $N$-configurations are ``covered'' in this way. But this follows from Lemmas \ref{lemma_4} and \ref{lemma_6}. 

Now, the case $n > p$ follows likewise by commuting the roles of $P$ and $N$. Note that the terms $\arg \max \{y^0_j / \pi_{dj} \,\big|\; j \in N\}$ and $(y^0_1,\dots y^0_{d-1}, y^0_k / \pi_{dk})^\top$ have to be replaced, respectively, by $\arg \min \{\pi_{id} y^0_i \,|\; i \in P\}$ and $(y^0_1,\dots y^0_{d-1}, \pi_{kd} y^0_k)^\top$. This completes the induction argument.
\end{proof}

Let us continue the interpretation given before Example \ref{ex:1}: Consider a bipartition $(P,N)$ and a consistent price system $y \in K_d^+\setminus\cb{0}$ which is realizable by only transactions on arcs of a spanning tree $T$. Then the degree vector $\deg_T(P)$ describes a {\em contribution scheme}: Any ``positive'' asset $i \in P$ in a portfolio is used in order to buy shares of $\deg_T(i)$ different ``negative'' assets $j \in N$. Theorem \ref{theorem_7} states that for every given contribution scheme $c \in \PP$, there exists a realizable consistent price system $y$.

\begin{corollary} \label{corollary_9}
Assume that also \eqref{equation_5} holds. Let $x,y$ be two feasible tree solutions with respect to bipartitions $(P_x,N_x)$ and $(P_y,N_y)$ of $V$, respectively. Then $(P_x,N_x) \neq (P_y,N_y)$ implies $x \neq \alpha y$ for all $\alpha > 0$. Moreover, $K_d^+$ has at least $2^d-2$ extreme directions.
\end{corollary}	
\begin{proof}
Let $x,y$ be generated by the spanning trees $T_x,T_y$ of $G(P_x,N_x)$ and $G(P_y,N_y)$, respectively. From $(P_x,N_x) \neq (P_y,N_y)$ we obtain $(P_y\cap N_x)\cup(P_x\cap N_y)\neq\emptyset$. Without loss of generality, we assume there exists $i \in (P_y\cap N_x)$. There is $k \in P_x$ and $j \in N_y$ such that $ki \in E(T_x)$ and $ij \in E(T_y)$. Hence we have $\pi_{ki} x_k = x_i$ and $\pi_{ij} y_i = y_j$. Assume there is $\alpha > 0$ such that $x=\alpha y$, then we obtain $\pi_{ki} y_k = y_i$. Thus \eqref{equation_5} implies $\pi_{kj}y_k < \pi_{ki}\pi_{ij} y_k = y_j$. By \eqref{equation_6}, we obtain $y \not\in K_d^+$, a contradiction. The second statement follows as there are $\sum_{p=1}^{d-1} \binom{d}{p} = 2^d - 2$ different bipartitions of $V=\cb{1,\dots,d}$.
\end{proof}

\begin{corollary}\label{corollary_10}
$K_d^+$ has at most $\sum_{p=1}^{d-1} \binom{d-2}{p-1}\binom{d}{p}$ extreme directions.
\end{corollary}
\begin{proof}
Consider a bipartition $(P,N)$ and set $p=|P|$. Then there are $\binom{d-2}{p-1}$ different $P$-configurations for $(P,N)$. By Lemma \ref{lemma_6} there are at most $\binom{d-2}{p-1}$ feasible tree solutions with respect to $(P,N)$ which are pairwise different, i.~e. $x \neq \alpha y$ for all $\alpha > 0$. There are $\binom{d}{p}$ different bipartitions with $|P|=p$, and $\sum_{p=1}^{d-1} \binom{d}{p}$ bipartitions in total. 
The statement now follows from Theorem \ref{theorem_3}.
\end{proof}

The following example shows that the bound in Corollary \ref{corollary_10} cannot be improved.

\begin{example}\label{example_11}
Let the non-diagonal entries of the matrix $\Pi \in \R^{d\times d}$ ($d \geq 2$) be pairwise different prime numbers such that $\of{\min\cb{\pi_{ij}\st ij\in V\times V, i\neq j}}^2 > \max\cb{\pi_{ij}\st ij\in V\times V, i\neq j}$. This is possible for arbitrary dimension by the prime number theorem. Further let $\pi_{ii}:=1$ for all $i \in V$.
Then, \eqref{equation_1}, \eqref{equation_2} and \eqref{equation_5} are satisfied.
Corollary \ref{corollary_9} and Theorem \ref{theorem_3} yield that we obtain pairwise
different extreme directions for pairwise different bipartitions. 

Let us fix an arbitrary bipartition $(P,N)$. Assume that $|\PP(y)| > 1$ for some feasible tree solution $y$ with respect to $(P,N)$. Then $H(y)$ contains a cycle $C=\cb{e_1,\dots,e_{2k}}$ and we obtain $\prod_{l=1}^k \pi_{2l-1} = \prod_{l=1}^k \pi_{2l}$. But the $\pi_{ij}$ are pairwise different prime numbers, whence the contradiction. Thus we have $|\PP(y)| = 1$ for all feasible tree solutions. 

Given two $P$-configurations $c^1,c^2 \in \N^P$ such that $c^1 \neq c^2$. By Theorem \ref{theorem_7} there are feasible tree solutions $y^1,y^2$ such that $c^1 \in \PP(y^1)$ and $c^2 \in \PP(y^2)$. Clearly, we have $\PP(y)=\PP(\alpha y)$ for all $\alpha>0$. Using $|\PP(y^1)|=|\PP(y^2)|=1$ we conclude that $y^1 \neq \alpha y^2$ for all $\alpha > 0$. This means that we obtain pairwise different extreme directions of $K^+_d$ from pairwise different $P$-configurations. Hence, in this example,  $K_d^+$ has exactly $\sum_{p=1}^{d-1} \binom{d-2}{p-1}\binom{d}{p}$ extreme directions.
\end{example}

The proof of Theorem \ref{theorem_7} provides a method to compute selected extreme directions, i.~e. extreme directions with respect to a given bipartition $(P,N)$ and a given $P$-configuration $c$ (or a given $N$-configuration $b$). This allows to compute certain subsets of generating vectors in cases where the computation of all extreme directions is not any more tractable. The method consists of the two functions given below. Each function calls itself or the other function recursively. Let $\PP$ and $\NN$ denote, respectively, the set of all $P$-configurations and the set of all $N$-configurations with respect to a bipartition $(P,N)$. To explains the main idea of the algorithm, we start with an example.

\begin{example} Let $V=\cb{1,2,3,4,5}$, $P=\cb{1,4}$ and $N=\cb{2,3,5}$. Assume that we want to compute a feasible tree solution $y$ that is generated by a spanning tree $T$ with $\deg_T(P) = (2,2) = (c_1,c_4) \in \N^P$. As all components of $c$ are greater than $1$, we have to consider all $N$-configurations, that is, $(2,1,1)$, $(1,2,1)$, and $(1,1,2)$. In the first case, $b=(b_2,b_3,b_5)=(2,1,1)$, we delete node $3$ and node $5$ as they have degree $1$. For the resulting digraph with nodes $1,4,2$ and arcs $(1,2)$, $(4,2)$, a feasible tree solution is $y_2 = 1$, $y_1=\pi_{12}^{-1}$, $y_4=\pi_{42}^{-1}$. Node $3$ is added together with one of the arcs $(1,3)$ and $(4,3)$. The first one is chosen
if $\pi_{13} y_1 < \pi_{43} y_4$. In this case, we set $y_3 = \pi_{13}\pi_{12}^{-1}$. Otherwise, the arc $(4,3)$ is chosen and we set $y_3 = \pi_{43}\pi_{42}^{-1}$. This rule ensures feasibility of the constructed $y$. Likewise we add node $5$ together with one of the arcs $(1,5)$ or $(4,5)$, which is chosen by the same rule. We obtain a spanning tree $T$ and a corresponding feasible tree solution $y$ such that $\deg_T(P)$ equals one of the three $P$-configurations $(3,1)$, $(2,2)$, $(1,3)$. The remaining two cases $b=(1,2,1)$ and $b=(1,1,2)$ are treated likewise. Among the three spanning trees (with corresponding feasible tree solutions) there is one with $\deg_T(P) = (2,2)$.
\end{example}

\floatname{algorithm}{Algorithm part}
\begin{algorithm}\label{algorithm_1} 
	\caption{function $(y,b) = \text{getb}(P,N,c)$; 
	\newline\textit{Input:} bipartition $(P,N)$, $P$-configuration $c$; 
	\newline\textit{Output:} feasible tree solution $y$ with resp. to $(P,N)$; $N$-configuration $b \in \NN(y)$; } 
	\begin{algorithmic}		
		\IF{$|P|=1$}
		   \STATE $\forall i \in P, j \in N:\; y_i \leftarrow 1, y_j \leftarrow \pi_{ij}, b_j \leftarrow 1$;
		\ELSIF{$P^0 := \of{i \in P \st c_i \neq 1} \neq P$ \textbf{and} $P^0 \neq \emptyset$ }
			\STATE $c^0\leftarrow c(P^0):=\of{c_i \st i \in P^0}$;
			\STATE $(y^0,b) \leftarrow \text{getb}(P^0,N, c^0)$; (Recursion: solve a smaller problem)
			\STATE $\forall i \in P^0: y_i \leftarrow y^0_i$; $\forall i \in P \setminus P^0:  y_i \leftarrow y^0_{k(i)}/ \pi_{i k(i)} \text{ where } k(i) := \arg \max \{y^0_j /\pi_{ij} \,|\; j \in N\}$;
			\STATE $\forall  i \in P \setminus P^0: b_{k(i)} \leftarrow b_{k(i)} + 1$;
		\ELSE
		    \FOR{ $b \in \NN$}
				\STATE $(y,\bar c) \leftarrow \text{getc}(P,N,b)$; (Recursion: solve ``complementary'' problems, see part 2)
				\STATE \textbf{if } $c=\bar c$ \textbf{ then } break; \textbf{ end if}	
			\ENDFOR
		\ENDIF
	\end{algorithmic}
\end{algorithm}	

\floatname{algorithm}{Algorithm part}
\begin{algorithm}\label{alg2} 
	\caption{function $(y,c) = \text{getc}(P,N,b)$; 
	\newline\textit{Input:} bipartition $(P,N)$, $N$-configuration $b$; 
	\newline\textit{Output:} feasible tree solution $y$ with resp. to $(P,N)$; $P$-configuration $c \in \PP(y)$; }
	\begin{algorithmic}
		\IF{$|N|=1$}
		   \STATE $\forall i \in P, j \in N:\; y_j \leftarrow 1, y_i \leftarrow \frac{1}{\pi_{ij}}, c_i \leftarrow 1$;
		\ELSIF{$N^0 := \of{j \in N \st b_j \neq 1} \neq N$ \textbf{and} $N^0\neq \emptyset$}
			\STATE $b^0 \leftarrow b(N^0):=\of{b_j \st j \in N^0}$;
			\STATE $(y^0,c) \leftarrow \text{getc}(P,N^0, b^0)$; (Recursion: solve a smaller problem)
			\STATE $\forall j \in N^0: y_j \leftarrow y^0_j$; $\forall j \in N \setminus N^0:  y_j \leftarrow \pi_{k(j) j} y^0_{k(j)} \text{ where } k(j) := \arg \min \{\pi_{ij} y^0_i \,|\; i \in P\}$;
			\STATE $\forall  j \in N \setminus N^0: c_{k(j)} \leftarrow c_{k(j)} + 1$;
		\ELSE
		    \FOR{ $c \in \PP$}
				\STATE $(y,\bar b) \leftarrow \text{getb}(P,N,c)$; (Recursion: solve ``complementary'' problems, see part 1)
				\STATE \textbf{if } $b=\bar b$ \textbf{ then } break; \textbf{ end if}	
			\ENDFOR
		\ENDIF
		\end{algorithmic}
	\end{algorithm}
	
We continue with a larger example which was solved on a computer.

\begin{example}\label{example_12} Let $d=20$ and let the matrix $\Pi$ be defined as follows: The diagonal entries are set to $1$. The non-diagonal entries are filled line-wise with consecutive prime numbers starting with $\pi_{12}=59$. Then the largest entry is $\pi_{20,19}=2713$. Of course, \eqref{equation_1} and \eqref{equation_2} hold. Since $59^2>2713$, \eqref{equation_5} is satisfied. The arguments used in Example \ref{example_11} yield that $K^+_{20}$ has exactly $\sum_{p=1}^{19} \binom{18}{p-1}\binom{20}{p}= 35.345.263.800$ extreme directions. Let us choose a bipartition $(P,N)$ with $P=\cb{5,6,7,8,9,10,11}$ and $N=\cb{1,\dots,4,12,\dots,20}$. There are $\binom{18}{6}=18564$ different $P$-configurations for this bipartition. Take for instance the $P$-configuration $c=\of{3,2,4,2,2,2,4}^\top \in \N^P$. Then the algorithm yields the extreme direction
	\begin{align*}
	y=&\left(
	 \frac{487\cdot 757}{503\cdot 859},
	 \frac{491\cdot 757}{503\cdot 859},
	 \frac{619\cdot 947\cdot 1367}{677\cdot 953\cdot 1427},
	 \frac{757}{859},
	 \frac{757}{503\cdot 859},
	 \frac{947\cdot 1367}{677\cdot 953\cdot 1427},
	 \frac{1}{859},
	 \frac{1367}{953\cdot 1427},\right.\\
	 &\;\;\;\frac{1}{1117},
	 \frac{839}{859\cdot 1237},
	 \left.\frac{1}{1427},
	 \frac{1327}{1427},
	 \frac{947\cdot 1367}{953\cdot 1427},
	 \frac{1367}{1427},
	 \frac{1373}{1427},
	 \frac{829}{859},
	 \frac{839}{859},
	 \frac{839\cdot 1249}{859\cdot 1237},
	 \frac{1109}{1117},
	 1\right)^\top 
	\end{align*}
and the $N$-configuration $b=\of{1,1,1,2,1,2,2,1,1,2,1,1,3}^\top \in \N^N$	as the result.
	
\end{example}	

Next we consider the special case where $\pi_{ii}:=1$ and $\pi_{ij} := a_j/b_i$ ($i \neq j$) for given vectors $a, b \in \R^V$ such that $0 < b_i \leq a_i$ for all $i \in V$ and $0 < b_k < a_k$ for at least one $k \in V$.
It follows that conditions \eqref{equation_1} to \eqref{equation_4} hold. If we additionally assume that
\begin{equation}\label{equation_9}
	\forall i \in V:\; 0 < b_i < a_i,	
\end{equation}
then \eqref{equation_5} is satisfied, too. 


Let us give two situations that lead to this special case: In the first situation, we are given bid prices $b_i$ and ask prices $a_i$ for each asset $i\in V$, all denoted in the same currency, and a transaction between any two distinct assets
can only be made via cash in this currency (and not directly). If one of the assets, say asset $k$, is the currency itself, then $b_k=a_k=1$ and thus \eqref{equation_9} does not hold. If all assets (including possibly a riskless asset like a bond) have a positive bid-ask spread, then  \eqref{equation_9} is satisfied.
Another situation leading to the special case is a market with $d$ currencies having frictionless prices $(1,S_2,\dots,S_d)$, all expressed in currency $1$. Whenever currency $i$ is exchanged into currency $j \neq i$, transaction costs are charged at a fixed rate $k>0$ against asset $i$. Then, $\pi_{ij}=\frac{a_j}{b_i}$ for $i\neq j$ with $a_j = (1+k)S_j$ and $b_i = S_i$ and \eqref{equation_9} holds.

The following consideration tells us that in the special case considered here the realizable consistent price systems are independent from the contribution schemes, i.e., there is at most one extreme direction of $K_d^+$ for each bipartition:

Consider a feasible tree solution $y$ generated by a spanning tree $T$. If an arc $kl$ of $G$ does not belong to $T$, one can consider a path in $T$ from $k$ to $l$ such that \eqref{equation_7} holds along this path. Using the special structure   
of the $\pi_{ij}$'s, we obtain that $\pi_{kl} y_k = y_l$. This means that $y$ does not depend on the choice of the spanning tree $T$ of $G(P,N)$. The following result is immediate. It shows that, in this special case, we do not need the above algorithm as the generating vectors of $K_d^+$ can be given explicitly.

\begin{corollary} \label{corollary_13}
For the special case $\pi_{ij} = a_j/b_i$, $K_d^+$ can be expressed as
$$ K_d^+ =\cone\cb{y \in \R^d \st (P,N) \text{ is a bipartition of } V,\; \forall i \in P:\, y_i = b_i,\; \forall j \in N:\, y_j=a_j}.$$
The solvency cone $K_d$ for the special case $\pi_{ij} = a_j/b_i$ can be expressed by
$$ K_d=\big\{x \in \R^d \st \forall \text{ bipartitions } (P,N) \text{ of } V:\; \textstyle\sum_{i \in P} b_i x_i + \textstyle\sum_{j \in N} a_j x_j \geq 0\big\}.$$
$K_d^+$ has at most $2^d - 2$ extreme directions.
\end{corollary}

\begin{corollary} \label{corollary_14} Let \eqref{equation_9} be satisfied for the special case $\pi_{ij} = a_j/b_i$. Then, there is a one-to-one map between the set of all bipartitions of $V$ and the set of all extreme directions $y$ of $K_d^+$, which can be expressed as $y_i= b_i$ for $i \in P$ and $y_j=a_j$ for $j \in N$. $K_d^+$ has exactly $2^d - 2$ extreme directions. 
\end{corollary}
\begin{proof}
	Follows from Corollary \ref{corollary_9} and the above considerations.
\end{proof}

\begin{remark}
	From Corollary \ref{corollary_13} we can derive a recursive formula for the matrix $Y_d$ of generating vectors of $K_d^+$ (as columns) for the special case $\pi_{ij} = a_j/b_i$. For all $d \geq 3$ we have 
$$ Y_2 = \begin{pmatrix}
           a_1 & b_1 \\
		   b_2 & a_2
         \end{pmatrix} \hspace{2cm}
 Y_d = \begin{pmatrix}
					 	 &          &     & b_1       &            &          &      & a_1     \\
	                     & Y_{d-1}  &     & \vdots    &            &  Y_{d-1} &      & \vdots  \\
	                     &          &     & b_{d-1}   &            &          &      & a_{d-1} \\
		             a_d & \dots    & a_d & a_d       &  b_d       & \dots    & b_d  & b_d 
	   \end{pmatrix}.$$ 
If \eqref{equation_9} is satisfied, all columns of $Y_d$ are extreme directions of $K_d^+$. Then, there is no redundant generating vector in $Y_d$.	   
\end{remark}

We now consider the case where \eqref{equation_9} does not hold in the sense that $b_k=a_k$ for some $k \in V$.
 
\begin{proposition} \label{proposition_15}
For the special case $\pi_{ij} = a_j/b_i$ with $b_k=a_k$ for some $k \in V$, $K_d$ can be expressed as
$$ K_d = \cone\cb{a_j e^i - b_i e^j \st ij \in (V \times \cb{k}) \cup (\cb{k}\times V)}.$$	
\end{proposition}
\begin{proof}
	Let $ij \in V \times V$ such that $i \neq k$ and $j\neq k$. Then 
	$$\pi_{ij} e^i - e^j = a_j \of{\frac{1}{b_i} e^i - \frac{1}{a_k} e^k + \frac{1}{b_k} e^k - \frac{1}{a_j} e^j} 
	                     =  \frac{a_j}{a_k} (\pi_{ik} e^i - e^k ) + (\pi_{kj} e^k - e^j)$$
is not an extreme direction and can be omitted in the definition of $K_d$.	
\end{proof}

\begin{corollary} \label{corollary_16} For the special case $\pi_{ij} = a_j/b_i$ with $b_k=a_k$ for some $k \in V$, one has
$$ K_d^+ =\cone\cb{y \in \R^d \st Q \subseteq V \setminus\cb{k},\; \forall i \in Q:\, y_i = b_i,\; \forall j \in V \setminus Q:\, y_j=a_j}.$$	
$K_d^+$ has at most $2^{d-1}$ extreme directions. If the inequalities in \eqref{equation_9} hold for all $i \in V\setminus\cb{k}$, then $K_d^+$ has exactly $2^{d-1}$ extreme directions.
\end{corollary}
\begin{proof} Without loss of generality we can assume $k=d$ and $b_d=a_d=1$. The generating vectors of $K_d^+$ are given by the columns of the matrix
$$ Y_d = \begin{pmatrix}
					 	           &&   & b_1 & a_1 \\
	                     &Y_{d-1}&  & \vdots    & \vdots     \\
	                     	 	   &&  & b_{d-1} & a_{d-1} \\
		                 1 & \dots & 1 & 1 &    1\\
	           \end{pmatrix},$$
where $Y_{d-1}$ is the matrix whose columns are the $2^{d-1}-2$ generating vectors of $K_{d-1}^+$ arising from the characterization given in Corollary \ref{corollary_13}. The first two statements now follow from Corollary \ref{corollary_13}. The last assertion follows from Corollary \ref{corollary_14} if we can show that $\bar b :=(b_1,\dots,b_{d-1},1)^\top$ and $\bar a := (a_1,\dots,a_{d-1},1)^\top$ are extreme directions, and if the vectors $(y,1)^\top$ for  columns $y$ of $Y_{d-1}$ cannot be expressed as convex combinations of $\bar b$ and $\bar a$. Assume that $\bar b$ is not an extreme direction, then $(b_1,\dots,b_{d-1})^\top$ must be a convex combination of columns of $Y_{d-1}$. But this is not possible since $b_i \leq y_i$ for all columns of $Y_{d-1}$ with at least one strict inequality. Likewise we show the remaining cases.
\end{proof}

\begin{remark} \label{remark_22}
	From Corollary \ref{corollary_16} we also obtain a recursive definition of the matrix $Y_d$ of generating vectors of $K_d^+$ (as columns) for the special case considered there. Without loss of generality we can assume $b_1=a_1=1$, which can be seen as the case where all bid and ask prices $b_i,a_i$ for $i \in \cb{2,\dots,d}$ are expressed by asset $1$ (``cash''). For all $d \geq 3$ we have 
$$ Y_2 = \begin{pmatrix}
           1   &   1 \\
		   a_2 & b_2
         \end{pmatrix} \hspace{2cm}
 Y_d = \begin{pmatrix}
					 	 &          &           &            &          &     \\
	                     & Y_{d-1}  &           &            &  Y_{d-1} &     \\
	                     &          &           &            &          &     \\
		             a_d & \dots    & a_d       &  b_d       & \dots    & b_d 
	   \end{pmatrix}.$$ 
\end{remark}

\noindent
In case of $a_k=b_k$ for more than one index $k$, one can easily see that half of the generating vectors become duplicates and can be deleted for each additional equation $a_k=b_k$.

\appendix
\section*{Appendix}

Here we provide the details of the proof of Lemma \ref{lemma_4} using the results of Postnikov \cite{Postnikov09}. We set $[n]:=\cb{1,\dots,n}$ and denote by $K_{m,n}$ the complete bipartite graph with $m$ left nodes and $n$ right nodes. For polytopes $A,B$ in $\R^n$ we define the Minkowski sum and the Minkowski difference as
	$$A+B:=\cb{a+b|\;a\in A,\; b\in B} \qquad A-B := \cb{x \in \R^n |\; \cb{x} + B \subseteq A}.$$
Let $G \subseteq K_{m,n}$ be a bipartite graph with no isolated vertices. For the graph $H$ in Lemma \ref{lemma_4} this can be assumed since otherwise there is no spanning tree. The vertices of $G$ are labeled  by $1,\dots,m$ (called left vertices) and $\bar 1,\dots,\bar n$ (called right vertices). The mirror image of $G$ is the bipartite graph $G^* \subseteq K_{n,m}$ obtained from $G$ by switching the left and right components. The graph $G$ is associated with the collection $\mathcal{I}_G$ of subsets $I_1,...,I_m \subseteq [n]$ such that $j \in I_i$ if and only if $(i,\bar j)$ is an edge of $G$. Consider the polytope 
$$ P_G^- := (\Delta_{I_1} + \dots + \Delta_{I_m}) -\Delta_{[n]},$$
	where $\Delta_I := \conv\cb{e^i |\; i \in I}$ for $I\subseteq [n]$ and $e^i$ denoting the $i$-th unit vector in $\R^n$. In \cite{Postnikov09} (compare Definition 11.2 and the beginning of Section 9), 
	the polytope $P_G^-$ is called {\em trimmed generalized permutohedron} and is denoted by $P_G^-(1,\dots,1)$. A sequence of nonnegative integers $(a_1,\dots,a_m)$ is called a {\em $G$-draconian sequence} \cite[Definition 9.2]{Postnikov09} if $a_1 + \dots + a_m = n-1$ and, for any subset 
	$\cb{i_1 < \dots\ < i_k} \subseteq [m]$, we
	have $|I_{i_1} \cup \dots \cup I_{i_k}| \geq a_{i_1} + \dots + a_{i_k} + 1$.

	Theorem 11.3 in \cite{Postnikov09} states that the number of lattice points (i.~e. points in $\mathbb{Z}^n$) in $P_G^-$ equals the number of $G$-draconian sequences (and likewise for $G^*$). Corollary 11.8 in \cite{Postnikov09} states that the number of lattice points of $P_G^-$ equals the number of lattice points in $P_{G^*}^-$. As a consequence, the number of $G$-draconian sequences equals the number of $G^*$-draconian sequences.

	To prove Lemma \ref{lemma_4}, it remains to show that $(a_1,\dots,a_m)$ is a $G$-draconian sequence if and only if $(a_1+1,\dots,a_m+1)$ is a left degree sequence of a spanning tree of $G$. This is done in the following two propositions. The first one is summary of the following items of \cite{Postnikov09}: 
	\begin{itemize}
	\item the equivalent characterization of a $G$-draconian sequence given in Definition 9.2 using the first of the three equivalent conditions in Proposition 5.4,
	\item the proof that in Proposition 5.4., the first condition implies the third condition (which was left as an exercise),
	\item a reformulation of the third condition of Proposition 5.4 in terms of a graph $F$.
	\end{itemize}

	\begin{proposition} \label{p2}
	If $(a_1,\dots,a_m)$ is a $G$-draconian sequence, then the graph $F \subseteq K_{n-1,n}$ associated with the sets $I_1^{(a_1)}, \dots ,I_m^{(a_m)}$, where $I_i^{(a_i)}$ means $I_i \in \mathcal{I}_G$ repeated $a_i$ times, has a spanning tree $T$ with $\deg_T(i) = 2$ for all left vertices $i$ of $F$.
	\end{proposition}
	\begin{proof} Let us denote the sets $I_1^{(a_1)}, \dots ,I_m^{(a_m)}$ by $J_1,\dots,J_{n-1}$. Since $(a_1,\dots,a_m)$ is a $G$-draconian sequence, we conclude that, for any distinct $i_1,\dots,i_k$, we have 
	\begin{equation}\label{eq1}
		|J_{i_1} \cup \dots \cup J_{i_k}| \geq k+1.
	\end{equation}
	 In particular, all left nodes of $F$ have a degree of at least $2$. Consider all subgraphs of $F$ whose left nodes have degree $2$. Among these subgraphs, consider those with the minimal number of components, say $r$ components.  Assume that $r \geq 2$. This implies that at least one component contains a cycle (recall that $|E|=|V|-\kappa$, where $\kappa$ is the number of components, if and only if the graph is acyclic). Among all subgraphs with left degrees 2 and exactly $r$ components choose one, say the subgraph $S$, which contains a cycle $C$ with a maximal number of vertices. Denote by $K$ the component of $S$ that contains $C$. Of course, the cycle $C$ has the same number of left and right vertices. By \eqref{eq1} (applied to the left nodes of $C$), there is a left node $v$ in $C$ and an edge $(v,w) \in E(F) \setminus E(S)$. We distinguish two cases: (i) If $w \not\in K$, we add $(v,w)$ to $S$ and delete an edge $(v,u)$ in $C$. This maintains the degree condition and we obtain a subgraph with less than $r$ components, a contradiction. (ii) If $w \in K$, adding $(v,w)$ to $S$ creates a cycle $D$ that contains $(v,w)$. The degree condition ensures that there is an edge $(v,u) \in E(C) \cap E(D)$. Deleting $(v,u)$ produces a cycle having more vertices than $C$. This contradicts the maximality of $C$. Hence we obtain $r=1$, i.~e., we have found a desired spanning tree $T$.
	\end{proof}

	\begin{proposition}
	$(a_1,\dots,a_m)$ is a $G$-draconian sequence if and only if $(a_1+1,\dots,a_m+1)$ is a left degree sequence of a spanning tree of $G$.
	\end{proposition}
	\begin{proof}
	Let $a_i := \deg_T(i)-1$ for all left vertices $i \in [m]$ of a spanning tree $T$ of $G$. Assume that $(a_1,\dots,a_m)$ is not $G$-draconian, hence it is not $T$-draconian. Since $T$ has $n+m-1$ edges, we have $n-1 =\sum_{i=1}^m \deg_T(i) - m = \sum_{i=1}^m a_i$. Thus there are $\cb{i_1 < \dots\ < i_k} \subseteq [m]$ and sets $I_{i_1}, \dots, I_{i_k}$ associated with a subgraph $S$ of $T$ such that
	$|I_{i_1} \cup \dots \cup I_{i_k}| < a_{i_1} + \dots + a_{i_k} + 1$. We obtain $|V(S)| = |I_{i_1} \cup \dots \cup I_{i_k}| + k \leq \deg_T(i_1) + \dots + \deg_T(i_k) = |E(S)|$. Hence $S$ has a cycle, which is a contradiction.

	To show the opposite implication, consider the graph $F$ and its spanning tree $T$ constructed in Proposition \ref{p2}. In order to obtain the desired spanning tree $T_G$ of $G$, for every $i \in [m]$, we merge successively the $a_i$ left nodes in $F$ that belong to the same $I_i$. Merging two nodes to a new node $j$, we distinguish two cases: (i) The number of edges remains unchanged. In this case we create a cycle and we can delete one edge of it such that the degree of $j$ reduces by one. (ii) The number of edges reduces by one (as merging two nodes may imply that two edges are merged to one). Note that the number of edges cannot be reduced by more than one as the graph does not contain a cycle at the beginning of each step. This procedure produces a tree $T_G$ of $G$ such that $a_i = \deg_{T_G}(i) - 1$, whenever $a_i > 0$. Adding one appropriate edge to each left node $i$ with $a_i = 0$, we obtain the desired spanning tree.	
\end{proof}

\end{document}